\documentclass[12pt,a4paper]{article}

\advance \textwidth -60truemm\relax

\advance\oddsidemargin -1truein\relax

\evensidemargin=\oddsidemargin

\advance\textheight -60truemm\relax
\advance\topmargin -1truein\relax
\unitlength\textwidth
\divide\unitlength by 150\relax

\usepackage{amsmath,amssymb}
\usepackage{txfonts}
\usepackage{bm}
\bmdefine{\aaa}{a}
\bmdefine{\bbb}{b}
\bmdefine{\ccc}{c}
\bmdefine{\mmm}{m}
\bmdefine{\ppp}{p}
\bmdefine{\qqq}{q}
\bmdefine{\uuu}{u}
\bmdefine{\vvv}{v}
\bmdefine{\www}{w}
\bmdefine{\eee}{e}
\bmdefine{\xxx}{x}
\bmdefine{\zerovec}{0}
\bmdefine{\onevec}{1}

\newcommand{\CCC}{\mathbb{C}}
\newcommand{\RRR}{\mathbb{R}}
\newcommand{\FFF}{\mathbb{F}}

\newcommand{\define}{\mathrel{:=}}

\newcommand{\diag}{{\mathrm{Diag}}}

\newcommand{\rank}{\mathrm{rank\,}}

\newcommand{\maxrank}{{\mathrm{max.rank}}}
\newcommand{\supp}{{\mathrm{supp}}}
\newcommand{\cof}{{\mathrm{Cof}}}

\numberwithin{equation}{section}
\newtheorem{thm}[equation]{Theorem}

\newtheorem{lemma}[equation]{Lemma}

\newtheorem{definition}[equation]{Definition}
\newtheorem{prop}[equation]{Proposition}

\newcommand{\bigzerou}{\smash{\lower1.7ex\hbox{\bg 0}}}

\newcommand{\bigastu}{\smash{\lower1.7ex\hbox{\bg *}}}

\begin{document}

\title{About the maximal rank of $3$-tensors over the real and the complex number field}
\author{Toshio Sumi, Mitsuhiro Miyazaki and Toshio Sakata%
\thanks{Kyushu University, Kyoto University of Education and Kyushu University}
}
\maketitle

\section{Introduction}
High dimensional array data, that is, tensor data,  are becoming 
important recently in various application fields (for example see 
Miwakeichi et al.~\cite{Miwa}, Vasilescu and Terzopoulos~\cite{Vasil-Terzo} and 
Muti and Bourennane~\cite{Muti-Bo}).  
A $p$-tensor is an element of $\FFF^{n_1}\otimes\FFF^{n_2}\otimes\cdots\otimes\FFF^{n_p}$,
where $\FFF$ is the real or complex number field and $n_1$, $n_2$, \ldots, $n_p$ are
positive integers.
It is known that every $p$-tensor can be expressed as a sum of $p$-tensors of the form
$a_1\otimes a_2\otimes \cdots\otimes a_p$.
The rank of a tensor $x$ is, by definition, the smallest number such that $x$ is
expressed as  a sum of the tensors of the above form.
Since there is a canonical basis in $\FFF^{n_1}\otimes\cdots\otimes\FFF^{n_p}$,
there is a one to one correspondence between the set of all $p$-tensors and the set of
$p$-dimensional arrays of elements of $\FFF$.
In particular, $3$-tensor can be identified to 
$A=(A_{1};A_{2};\cdots;A_{n_3})$,
where each $A_i$ is an $n_1\times n_2$ matrix.
The rank of a tensor may be considered to express complexity of the tensor. 
The factorization of a tensor to a sum of rank $1$ tensors means that the data 
is expressed by a sum of data with  most simpler structure, and 
we may have better understanding of data. This is an essential attitude 
for data analysis and therefore the problem of tensor factorization 
is an essential one for applications.  
 For modelling data, the maximal rank of {\lq\lq}a set of tensors{\rq\rq} (model)  
is also crucially important, because an excessive rank model is redundant and deficient 
rank model can not describe data fully. In this paper we consider the maximal rank 
problem of $3$-tensors. In the following by $T(a,b,c)$ or simply 
$\FFF^{a\times b \times c}$ we denote the set of all tensors with 
size $a \times b \times c$, and by $\maxrank_\FFF(a,b,c)$ denotes the maximal rank of all tensors in $T(a,b,c)$.
Note that in this paper 
$\FFF$ is $\CCC$, the complex number field, or $\RRR$, the real number filed.
Atkinson and Stephens~\cite{Atkinson-Stephens} and 
Atkinson and Lloyd~\cite{Atkinson-Lloyd}  developed a non-linear theory based on their 
own several lemmas. Basically they estimated the bounds by adding two diagonal matrices 
which enables the two matrices diagonalizable simultaneously. They did not solve the 
problem fully, and restricted the type of tensors for obtaining clear cut results. 
They obtained 
$\maxrank_{\CCC}(p,n,n) \leq (p+1)n/2$ 
for an even $p$ and $[p/2]n$ under the condition that 
$f(\lambda_{1},\ldots,\lambda_{p})=\det(\sum_{i=1}^{p}\lambda_{i}A_{i})$
is as a polynomial in $\CCC[\lambda_{1},\ldots,\lambda_{p}]$  not identically zero and 
has no repeated polynomial factor.  However they treated the problem 
over the complex number field. 
The aim of this paper is to give upper bound over the real number field.  
We traced their method and tried to rephrase their result to the real number field.  
It should be noted that the problem becomes difficult for the real field because the 
characteristic polynomial of a matrix dose not necessarily have real roots.  
In this paper we will report some generalization of 
Atkinson and Stephens~\cite{Atkinson-Stephens} and 
Atkinson and Lloyd~\cite{Atkinson-Lloyd}. 
In Section~\ref{sec:OverReal} we first consider the real versions of 
several lemmas treated in the complex number field in the two papers, 
and by which we show two main theorems, Theorem~\ref{thm:maxrankAK1} and 
Theorem~\ref{thm:AL:rank@n.n.p},  
which are slight extensions of Theorem 1 in 
Atkinson and Stephens~\cite{Atkinson-Stephens} 
and Theorem 1 in Atkinson and Lloyd~\cite{Atkinson-Lloyd} respectively. 
In Section~\ref{sec:maxrank@n.n.3},  we will prove the statement without proof given by 
Atkinson and Stephens~\cite{Atkinson-Stephens}:
$\maxrank_{\CCC}(n,n,3) \leq 2n-1$ and $\maxrank_{\CCC}(n,n+1,3) \leq 2n$. 
And we will prove the real version of these under some mild condition.
See Theorems \ref{thm:cont sing} and \ref{thm:non-square}.
As an application of this result,
we will prove, for the relatively small size of tensors from $T(3,3,3)$ to $T(6,6,3)$
the upper bound are given. 
We also give an upper bound for a more general type of tensors in 
$T(n,m,3)$ in case $n<m$: 
$\maxrank_{\FFF}(n,m,3) \leq n+m-1$.
This improves the result of  Atkinson-Stephens 
(see Theorem~\ref{thm:non-square}). 


\section{Preliminaries}

We first recall some basic facts and set terminology.

\begin{notation}
\begin{enumerate}
\item
By $\FFF$, we express the real number field $\RRR$ or the complex number field $\CCC$.
\item
For a tensor $x\in\FFF^m\otimes\FFF^n\otimes\FFF^p$ with
$x=\sum_{ijk}a_{ijk}\eee_i\otimes\eee_j\otimes\eee_k$,
we identify $x$ with
$(A_1;\cdots;A_p)$,
where $A_k=(a_{ijk})_{1\leq i\leq m,1\leq j\leq n}$ for $k=1$, \ldots, $p$ 
is an $m\times n$ matrix,
and call
$(A_1;\cdots;A_p)$ a tensor.
\item
For an $m\times n\times p$ tensor $T=(A_1;\cdots;A_p)$,
$l\times m$ matrix $P$
and
$n\times k$ matrix $Q$,
we denote by $PTQ$ the $l\times k\times p$ tensor
$(PA_1Q;\cdots;PA_pQ)$. 
\item
For an $m\times n\times p$ tensor $T=(A_1;\cdots;A_p)$,
we denote by $T^T$ the $n\times m\times p$ tensor
$(A_1^T;\cdots;A_p^T)$. 
\item
For $p$ $m\times n$ matrices $A_1$, \ldots, $A_p$,
we denote by $(A_1,  \ldots, A_p)$ the $m\times np$ matrix obtained by aligning
$A_1$, \ldots, $A_p$ horizontally.
\item
For $m\times n$ matrices $A_1$, \ldots, $A_p$,
we denote by 
$\langle A_1,\ldots, A_p\rangle$
the vector subspace spanned by $A_1$, \ldots, $A_p$ 
in the $\FFF$-vector space of all the $m\times n$ matrices
with entries in $\FFF$.
\item
For an $m\times n$ matrix $M$,
we denote the $m\times j$ 
(resp.\ $m\times (n-j)$)
matrix consisting of the first $j$ 
(resp.\ last $n-j$)
columns of $M$
by $M_{\leq j}$
(resp.\ ${}_{j<}M$).
We denote the $i\times n$ 
(resp.\ $(m-i)\times n$)
matrix consisting of the first $i$ 
(resp.\ last $m-i$)
rows of $M$
by $M^{\leq i}$
(resp.\ ${}^{i<}M$).
For integers $i_1$, \ldots, $i_r$ and $j_1$, \ldots, $j_s$ with
$1\leq i_1<\cdots <i_r\leq m$ and 
$1\leq j_1<\cdots <j_s\leq n$,
we denote the $r\times s$ matrix consisting of 
$i_1$-th, $i_2$-th, \ldots, $i_r$-th rows  and 
$j_1$-th, $j_2$-th, \ldots, $j_s$-th columns of $M$
by
$M^{=\{i_1,\ldots, i_r\}}_{=\{j_1,\ldots,j_s\}}$.
\item
We denote by $E_{ij}$ the matrix unit whose entry in $(i,j)$ cell
is $1$ and $0$ otherwise.
%
%
%
%
\end{enumerate}
\end{notation}

\begin{definition}\rm
Let $x$ be an element of $\FFF^m\otimes\FFF^n\otimes\FFF^p$.
We define the rank of $x$, denoted by $\rank x$, to be
$\min\{r\mid\exists \aaa_i\in\FFF^m$, $\exists\bbb_i\in\FFF^n$, $\exists\ccc_i\in\FFF^p$ for 
$i=1$, \ldots, $r$ such that $x=\sum_{i=1}^r\aaa_i\otimes\bbb_i\otimes\ccc_i\}$.
$\max\{\rank x\mid x\in\FFF^m\otimes\FFF^n\otimes\FFF^p\}$ is denoted by
$\maxrank_\FFF(m,n,p)$.
\end{definition}
It is clear from the definition that 
$\rank(x+y)\leq \rank x+\rank y$ for any
$x$, $y\in\FFF^m\otimes\FFF^n\otimes\FFF^p$.

\begin{definition}\rm
For a matrix $A=(a_{ij})$ we set
$\supp(A)\define\{(i,j)\mid a_{ij}\neq0\}$
and call it the support of $A$.
\end{definition}

The following lemmas are easily verified.

\begin{lemma}\label{lem:rank basic1}
Let $(A_1;\cdots;A_p)$ be an $m\times n\times p$ tensor.
Then $\rank(A_1;\cdots;A_p)=\min\{r\mid\exists C_1$, \ldots, $C_r$
such that $C_i$ is a rank 1 matrix and 
$\langle A_1,\ldots, A_p\rangle
\subset
\langle C_1,\ldots, C_r\rangle\}$.
In particular,
\begin{enumerate}
\item
if
$\langle A_1,\ldots, A_p\rangle =\langle B_1,\ldots, B_q\rangle$,
then $\rank(A_1;\cdots;A_p)=\rank(B_1;\cdots;B_q)$,
\item
for any non-singular matrices $P$ and $Q$ of size $m$ and $n$ respectively,
$\rank(A_1;\cdots;A_p)=
\rank(PA_1Q;\cdots;PA_pQ)$ and
\item
$\rank(A_1^T;\cdots;A_p^T)=\rank(A_1;\cdots;A_p)$.
\end{enumerate}
\end{lemma}

\begin{lemma}\label{lem:rank of tensor and matrix}
$\rank(A_1;\cdots;A_p)\geq\rank(A_1,\ldots,A_p)$.
\end{lemma}
From now on, we denote $\rank_\RRR$ or $\rank_\CCC$
instead of $\rank$ to specify over which field,
$\RRR$ or $\CCC$, we are working.
For the statements common to both fields, we use $\rank_\FFF$. 

The following lemma is well known.

\begin{lemma}\label{lem:distinct root}
Let
$$
f(\lambda)=\lambda^n+a_1\lambda^{n-1}+\cdots+a_n
$$
be a monic polynomial with a variable $\lambda$ and coefficients in $\FFF$.
Suppose that $f(\lambda)=0$ has $n$ distinct roots in $\FFF$.
Then there is a neighbourhood $U$ of $\aaa=(a_1,a_2,\ldots,a_n)^T$
in $\FFF^n$ such that for any $\xxx=(x_1,x_2,\ldots,x_n)^T\in U$,
$$
\lambda^n+x_1\lambda^{n-1}+\cdots+x_n=0
$$
has $n$ distinct roots in $\FFF$ and these roots are continuous function of $\xxx$.
\end{lemma}

\section{Maximal rank over the real number field\label{sec:OverReal}}
In this section we show results in the real number field which are 
obtained by Atkinson and Stephens~\cite{Atkinson-Stephens} and Atkinson and Lloyd~\cite{Atkinson-Lloyd} in the complex number field.  
We show the several results which is along with the results given by them, 
but the results are slightly different and some of them are new one.  
Now we prepare several lemmas which is a real version of Lemma in 
Atkinson-Stephens \cite{Atkinson-Stephens}. 
First we show the extended version of Lemma~3 in \cite{Atkinson-Stephens}.
%
%

\begin{lemma} \label{lem:as lemma3}
Let $A=(a_{ij})$ and $B=(b_{ij})$ be 
$n\times n$ matrices 
with entries in $\FFF$.
Then
there exist diagonal matrices  $X$, $Y$ 
with entries in $\FFF$
satisfying the followings.
\begin{enumerate}
\item \label{item:as lemma3 1}
$A+X$ is non-singular.
\item \label{item:as lemma3 2}
$(A+X)^{-1}(B+Y)$ has 
$n$ distinct eigenvalues in $\FFF$.
\end{enumerate}
Moreover 
if $i_1$, \ldots, $i_r$ are integers with
$1\leq i_1<\cdots<i_r\leq n$,
$
A_{=\{i_1,\ldots,i_r\}}^{=\{i_1,\ldots,i_r\}}
$
is non-singular and
$
(A_{=\{i_1,\ldots,i_r\}}^{=\{i_1,\ldots,i_r\}})^{-1}
(B_{=\{i_1,\ldots,i_r\}}^{=\{i_1,\ldots,i_r\}})
$
has $r$ distinct eigenvalues in $\FFF$,
then we can take $X$ and $Y$ so that the entries of the
$(i_u,i_u)$ cell of $X$ and $Y$ are zero for $u=1$, \ldots, $r$.
In particular,
\begin{itemize}
\item[(a)]
if
$(n,n)\in\supp(A)$,
then we can take $X$ and $Y$ so that the entries of the $(n,n)$ cell of $X$ and $Y$
are $0$.
\item[(b)]
if
$\{(n-1,n), (n,n-1)\}\subset\supp(A)$,
$(n,n)\not\in\supp(A)\cup\supp(B)$
 and $b_{n-1,n}/a_{n-1,n}\neq b_{n,n-1}/a_{n,n-1}$,
then we can take $X$ and $Y$ so that the entries of the $(n-1,n-1)$ and $(n,n)$ cells of $X$ and $Y$
are $0$.
\end{itemize}
\end{lemma}

\begin{proof}
First we prove the former half of the lemma.
Take distinct elements $s_1,\ldots, s_n$ of $\FFF$ and
set $D=\diag(s_1,\ldots,s_n)$.
Note that 
if 
the absolute values of
all entries of $A'$ 
are sufficiently small,
then
$A'+E_n$ is non-singular 
 and all entries of $(A'+E_n)^{-1}$ 
are continuous with respect to entries of $A'$.
Thus $(A'+E_n)^{-1}(B'+ D)$ is a continuous function with respect to
$A'$ and $B'$ if 
the absolute values of
their entries are sufficiently small.
Since 
$$\det(\lambda E_n-(A'+E_n)^{-1}(B'+ D))=0$$
has $n$ distinct 
roots $s_1$, $s_2$, \ldots, $s_n$ in $\FFF$
if $A'=B'=O$,
we see by Lemma \ref{lem:distinct root} that
there is a neighbourhood of $O$ 
in $\FFF^{n^2}$
such that if $A'$ and $B'$ are both in
it, then 
$$\det(\lambda E_n-(A'+E_n)^{-1}(B'+ D))=0$$
has $n$ distinct roots
in $\FFF$.
Hence for sufficiently small $\epsilon>0$, 
$$\det(\lambda E_n-(\epsilon A+E_n)^{-1}(\epsilon B+ D))=0$$
has $n$ distinct roots
in $\FFF$
and therefore
$$\det(\lambda (A+(1/\epsilon)E_n)-(B+(1/\epsilon)D))=0$$
has $n$ distinct roots
in $\FFF$.
So it is enough to set $X=(1/\epsilon)E_n$ and 
$Y=(1/\epsilon)D$.

Next we prove the latter half of the lemma.
By permuting the rows and columns simultaneously,
we may assume that $i_1=1$, \ldots, $i_r=r$.
Set
$$
A=
\begin{pmatrix}A_{11}&A_{12}\\A_{21}&A_{22}\end{pmatrix},
\quad
B=
\begin{pmatrix}B_{11}&B_{12}\\B_{21}&B_{22}\end{pmatrix},
$$
where $A_{11}$ and $B_{11}$ are $r\times r$ matrices.
Then, by assumption,
$A_{11}$ is non-singular and 
$(A_{11})^{-1}B_{11}$ 
has $r$ distinct eigenvalues,
say $s_1$, \ldots, $s_r$,
 in $\FFF$.
We take $n-r$ distinct elements 
$s_{r+1}$, \ldots, $s_n$ from $\FFF\setminus\{s_1$, \ldots, $s_r\}$
and set
$$
D_1=E_{n-r},
\quad
D_2=\diag(s_{r+1}, \ldots, s_{n}).
$$
Then by the same argument as in the proof of the former half,
we see that
$
\begin{pmatrix}
A_{11}&\epsilon A_{12}\\
\epsilon A_{21}&\epsilon^2A_{22}+D_1
\end{pmatrix}
$
is non-singular and
$$
\det\left(
\lambda E_n-
\begin{pmatrix}A_{11}&\epsilon A_{12}\\\epsilon A_{21}&\epsilon^2A_{22}+D_1\end{pmatrix}^{-1}
\begin{pmatrix}B_{11}&\epsilon B_{12}\\\epsilon B_{21}&\epsilon^2B_{22}+D_2\end{pmatrix}
\right)=0
$$
has $m$ distinct roots for sufficiently small $\epsilon>0$.
Therefore
$$
\det\left(
\lambda 
\begin{pmatrix}A_{11}&\epsilon A_{12}\\\epsilon A_{21}&\epsilon^2A_{22}+D_1\end{pmatrix}
-
\begin{pmatrix}B_{11}&\epsilon B_{12}\\\epsilon B_{21}&\epsilon^2B_{22}+D_2\end{pmatrix}
\right)=0
$$
has $m$ distinct roots for sufficiently small $\epsilon>0$.
Since
\begin{eqnarray*}
&&
\det\left(
\lambda 
\begin{pmatrix}A_{11}&A_{12}\\ A_{21}&A_{22}+\epsilon^{-2}D_1\end{pmatrix}
-
\begin{pmatrix}B_{11}&B_{12}\\ B_{21}&B_{22}+\epsilon^{-2}D_2\end{pmatrix}
\right)\\
&=&
\epsilon^{-2(n-r)}
\det\left(
\lambda 
\begin{pmatrix}A_{11}&\epsilon A_{12}\\\epsilon A_{21}&\epsilon^2A_{22}+D_1\end{pmatrix}
-
\begin{pmatrix}B_{11}&\epsilon B_{12}\\\epsilon B_{21}&\epsilon^2B_{22}+D_1\end{pmatrix}
\right),
\end{eqnarray*}
we see that it is enough to set
$X=\epsilon^{-2}\diag(O,D_1)$
and
$Y=\epsilon^{-2}\diag(O,D_2)$.
\end{proof}

The following result is well-known but we write a proof in convenience.

\begin{prop} \label{rank@a.b.ab}
If $n\geq ab$, it holds
$$\maxrank_{\FFF}(a,b,n)=ab.$$
\end{prop}

\begin{proof}
It is clear from the definition that
$\maxrank_\FFF(a,b,n)=\maxrank_\FFF(n,a,b)$.
If $A=(A_1;A_2;\cdots;A_b)$ is an $n\times a\times b$ tensor,
then it is also clear from the definition that
$\rank_\FFF A\geq \rank_\FFF (A_1,A_2,\ldots, A_b)$.
So we see that $\maxrank_\FFF(n,a,b)\geq ab$.

Next, let $A=(a_{ijk})$ be an arbitrary $3$-tensor.
Then
$$
A=\sum_{i=1}^a\sum_{j=1}^b\eee_i\otimes\eee_j\otimes(a_{ij1},a_{ij2},\ldots,a_{ijn})^T.
$$
Therefore, $\rank_\FFF A\leq ab$.
\end{proof}

We can show the real case of
Lemma~4 in \cite{Atkinson-Stephens}.

\begin{lemma}[cf. Lemma~4 \cite{Atkinson-Stephens}] \label{lem:AS:lemma4}
Let $X$ and  $Y$ be an $n \times n$ matrix such that  $X$ is non-singular 
and each root of $\det(\lambda X-Y)=0$ is in $\FFF$ and not repeated.
Then for any $n\times (m-n)$ matrices $U$ and $V$, it holds that  
$$\rank_{\FFF}(X,U;Y,V)\leq m.$$
\end{lemma}

\begin{proof}
We can apply the proof of Lemma~4 \cite{Atkinson-Stephens}.
\end{proof}

The following theorem is a slight generalization of Theorem~1 
in \cite{Atkinson-Stephens}.

\begin{thm} \label{thm:maxrankAK1}
Let $n\leq m$ and $\FFF=\RRR, \CCC$.
\begin{enumerate}
\item if $p$ is odd, it holds
$\displaystyle{ \maxrank_{\FFF}(n,m,p)\leq n+\frac{m(p-1)}{2}}$.
\item if $p$ is even, it holds
$\displaystyle{\maxrank_{\FFF}(n,m,p)\leq 2n+\frac{m(p-2)}{2}}$ and
in addition if $m=n$, it holds
$\displaystyle{\maxrank_{\FFF}(n,n,p)\leq \frac{n(p+2)}{2}-1}$.
\end{enumerate}
\end{thm}

\begin{proof}
Let $A=(A_1;\ldots;A_p) \in \FFF^{n\times m\times p}$.
There is non-singular matrices $P$ and $Q$ and integer $r\leq n$ such that
$PA_pQ=\begin{pmatrix} E_r&0\cr 0&0\end{pmatrix}$.
Then letting $B_j=PA_jQ$ for each $j=1,\ldots,p$, we have
$$\rank_{\FFF}(A_1;\cdots;A_{p})=\rank_{\FFF}(B_1;\cdots;B_p).$$
Let $D_p=B_p$ and $D_j=(D'_j,O)$ be $n\times m$ matrices with diagonal 
matrices $D'_j$ for $1\leq j<p$
such that $(B_{2i-1})_{\leq n}-D'_{2i-1}$ and
$(B_{2i})_{\leq n}-D'_{2i}$ 
satisfy the conditions of \ref{item:as lemma3 1} and \ref{item:as lemma3 2}
of Lemma \ref{lem:as lemma3}
for $i=1$, \ldots, $\lfloor(p-1)/2\rfloor$.
Then it holds 
$$\rank_{\FFF}(A) \leq 
  \rank_{\FFF}(D_1;\cdots;D_{p})
     +\rank_{\FFF}(B_1-D_1;\cdots;B_{p-1}-D_{p-1};O).$$
Thus for odd integer $i=1,3,5,\ldots$, 
we obtain
$\rank_{\FFF}(B_i-D_i;B_{i+1}-D_{i+1})\leq m$ 
by Lemma \ref{lem:AS:lemma4}.
Thus if $p$ is odd, we have
\begin{equation*}
\begin{split}
\rank_{\FFF}(A) & \leq n+\rank_{\FFF}(B_1-D_1;B_2-D_2)+ \cdots 
    + \rank_{\FFF}(B_{p-2}-D_{p-2};B_{p-1}-D_{p-1}) \\
 & \leq n+\frac{m(p-1)}{2}
\end{split}
\end{equation*}
and otherwise 
\begin{equation*}
\begin{split}
\rank_{\FFF}(A) & \leq n+\rank_{\FFF}(B_1-D_1;B_2-D_2)+ \cdots 
    + \rank_{\FFF}(B_{p-1}-D_{p-1};O) \\
 & \leq n+\frac{m(p-2)}{2}+n.
\end{split}
\end{equation*}
Furthermore, if $p$ is even and $m=n$, then
$\displaystyle{\rank_{\FFF}(A) \leq 2n+\frac{n(p-2)}{2}-1=\frac{n(p+2)}{2}-1}$
since we can take $D_{p-1}$ so that
$\rank(B_{p-1}-D_{p-1})\leq n-1$.
\end{proof}


Lemma~5 and Theorem~2 of \cite{Atkinson-Stephens} are also true over
the real number field whose proofs are quite similar.

\begin{lemma}[Lemma~5 \cite{Atkinson-Stephens}] \label{AK-Lemma5}
If $k\leq n$, then 
$$\maxrank_{\FFF}(m,n,mn-k)=m(n-k)+\maxrank_{\FFF}(m,k,mk-k).$$
\end{lemma}

\begin{thm}[Theorem~2 \cite{Atkinson-Stephens}]
If $k\leq m\leq n$, then
$$\maxrank_{\FFF}(m,n,mn-k)=mn-k^2+\maxrank_{\FFF}(k,k,k^2-k)$$
\end{thm}



Theorem~1 by Atkinson-Lloyd \cite{Atkinson-Lloyd} is also
slightly generalized.

\begin{thm} \label{thm:AL:rank@n.n.p}
Let $n\leq m$.
If $p$ is even, it holds
$$
\maxrank_{\FFF}(n,m,p)
   \leq 
\frac{m(p-1)}{2}
+n.
$$
\end{thm}

\begin{proof}
Let $A=(A_1;\cdots;A_p)\in\FFF^{n\times m\times p}$.
By 
\cite[Corollary~3.10]{Sumi-Miyazaki-Sakata},
there are tensor $T$ and non-singular matrices $P$ and $Q$ so that
$\rank_{\FFF}(T_1;T_2)\leq m/2$ and 
$P(A_p-T_1)Q$ and $P(A_{p-1}-T_2)Q$ are both of form $(D,O)$ with some 
diagonal matrix $D$.
Set $B_j=PA_jQ$ for $j=1,\ldots, p-2$, 
$D_{p-1}=P(A_{p-1}-T_2)Q$, and $D_{p}=P(A_p-T_1)Q$.
For diagonal matrices $D_{j}$ ($j=1,\ldots,p-2$), we have
\begin{equation*}
\begin{split}
\rank_{\FFF}(A)
  & \leq \rank_{\FFF}(B_1;\cdots;B_{p-2};D_{p-1};D_p)+\frac{m}{2} \\
  & \leq \rank_{\FFF}(B_1-D_1;\cdots;B_{p-2}-D_{p-2};O;O)+
    \rank_{\FFF}(D_1;\cdots;D_p)+\frac{m}{2} \\
  & \leq \sum_{j=1}^{(p-2)/2} 
       \rank_{\FFF}(B_{2j-1}-D_{2j-1};B_{2j}-D_{2j})+\frac{2n+m}{2}.
\end{split}
\end{equation*}
Thus by Lemmas \ref{lem:as lemma3} and \ref{lem:AS:lemma4},
we have
$$\rank_{\FFF}(A)
  \leq \frac{m(p-2)}{2}+\frac{2n+m}{2}
  =\frac{m(p-1)+2n}{2}
$$
 for some $D_1$, \ldots, $D_{p-2}$.
\end{proof}

\section{$\maxrank(m,n,3)$\label{sec:maxrank@n.n.3}}
In this section, we give a proof of the following statement 
(Theorem \ref{thm:as no proof})
asserted in
\cite{Atkinson-Stephens} without proof.
In fact, we prove more general statements over $\CCC$ 
and, under mild condition, over $\RRR$ also.
See Theorems \ref{thm:cont sing} and \ref{thm:non-square}.

\begin{thm}[\cite{Atkinson-Stephens}]\label{thm:as no proof}
$$ 
\maxrank_{\CCC}(n,n,3) \leq 2n-1 \text{ and } \maxrank_{\CCC}(n,n+1,3) \leq 2n.
$$ 
\end{thm}

We begin with the following lemma.

\begin{lemma}\label{lem:make nonzero}
Let $m$ be an integer with $m\geq 2$.
If $\aaa_1$, \ldots, $\aaa_s$, $\bbb_1$, \ldots, $\bbb_t$
are $m$-dimensional non-zero vectors
and
$A_1$, \ldots, $A_u$, $B_1$, \ldots, $B_v$
are $m\times 2$ matrices of rank $2$,
then there is a non-singular matrix $P$ such that
any entry of 
$P\aaa_i$ ($i=1$, \ldots, $s$),
$\bbb_i^T P^{-1}$ ($i=1$, \ldots, $t$)
and any 2-minor of
$PA_i$ ($i=1$, \ldots, $u$)
and
$B_i^T P^{-1}$ ($i=1$, \ldots, $v$)
is not zero.
\end{lemma}
\begin{proof}
Let $X=(x_{ij})$ be an $m\times m$ matrix of indeterminates,
i.e., $\{x_{ij}\}_{i,j=1}^m$ are independent indeterminates.
None of the following polynomials of $x_{ij}$ is zero,
where $\cof(X)$ is the matrix of cofactors of $X$.
\begin{itemize}
\item
$\det X$.
\item
$j$-th entry of $X\aaa_i$.
\item
$j$-th entry of $\bbb_i^T \cof(X)$.
\item
2-minor of $XA_i$ consisting of $j$-th and $k$-th
 rows
with $1\leq j< k\leq m$.
\item
2-minor of $B_i^T\cof(X)$ consisting of $j$-th and $k$-th 
columns 
with $1\leq j<k\leq m$.
\end{itemize}
So the product $f(x_{ij})$ of all the above polynomials  is not zero.
Since $\FFF$ is an infinite field,
we can take $p_{ij}\in\FFF$ so that $f(p_{ij})\neq 0$.
Then it is clear that
$P=(p_{ij})$ meets our needs since $P^{-1}=(\det P)^{-1}\cof(P)$.
\end{proof}
In order to estimate the rank of $n\times n\times 3$ tensors,
we prepare the following lemmas.
\begin{lemma}\label{lem:use ab}
Let $(A_1;A_2;A_3)$ be an $m\times n\times 3$ tensor with $m\leq n$ such that
$A_3=(D,O)$ where $D$ is a diagonal matrix with $0$ entry in $(m,m)$ cell
and
$(A_1)_{\leq m}$, $(A_2)_{\leq m}$ satisfy the condition of (a) or (b) of 
Lemma \ref{lem:as lemma3}.
Then
$\rank_\FFF(A_1;A_2;A_3)\leq m+n-1$.
\end{lemma}
\begin{proof}
By Lemma \ref{lem:as lemma3},
there are $m\times m$ diagonal matrices $D_1$ and $D_2$ with 0 entry in $(m,m)$ cell
such that
$(A_1+(D_1,O))_{\leq m}$ is non-singular and
$((A_1+(D_1,O))_{\leq m})^{-1}
((A_2+(D_2,O))_{\leq m})$ has $m$ distinct eigenvalues.
Therefore by Lemma \ref{lem:AS:lemma4}
\begin{eqnarray*}
&&\rank_\FFF(A_1;A_2;A_3)\\
&\leq&\rank_\FFF(A_1+(D_1,O);A_2+(D_2,O);O)+\rank_\FFF(-(D_1,O);-(D_2,O);A_3)\\
&\leq& n+m-1.
\end{eqnarray*}
\end{proof}

\begin{lemma}\label{lem:prep as lemma3}
Let $n$ be an integer with $n\geq 3$
and
$A_1$, $A_2$ $n\times n$ matrices with
$(n,n)\not\in\supp(A_1)\cup\supp(A_2)$.
Suppose that
$(A_1)_{=\{n\}}\neq\zerovec$ and $(A_1)^{=\{n\}}\neq\zerovec^T$
and 
for any $t\in\FFF$,
$(tA_1+A_2)_{=\{n\}}\neq\zerovec$ or 
$(tA_1+A_2)^{=\{n\}}\neq\zerovec^T$.
Then there is a non-singular $(n-1)\times (n-1)$ matrix $P$ such that
$A=\diag(P,1)A_1\diag(P,1)^{-1}$
and
$B=\diag(P,1)A_2\diag(P,1)^{-1}$
satisfy the condition of (b) in 
Lemma \ref{lem:as lemma3}.
\end{lemma}
\begin{proof}
Set
$A_1=
\begin{pmatrix}(A_1)_{\leq n-1}^{\leq n-1}&\aaa_1\\ \bbb_1^T&0\end{pmatrix}$
and
$A_2=
\begin{pmatrix}(A_2)_{\leq n-1}^{\leq n-1}&\aaa_2\\ \bbb_2^T&0\end{pmatrix}$.

First assume that
$\rank(\aaa_1,\aaa_2)=2$.
Then by Lemma \ref{lem:make nonzero}, we see that there is a non-singular
$(n-1)\times (n-1)$ matrix $Q_1$ such that any entry of
$Q_1\aaa_1$ and $\bbb_1^TQ_1^{-1}$ and any 2-minor of
$Q_1(\aaa_1,\aaa_2)$ is not zero.
Set
$Q_1(\aaa_1,\aaa_2)=(a_{ij})$ and 
$(\bbb_1,\bbb_2)^TQ^{-1}=(b_{ij})$.
If $(a_{n-1,1},a_{n-1,2})$ and $(b_{1,n-1}, b_{2,n-1})$ are linearly independent,
then $P=Q_1$ meets our needs since
$$
\diag(Q_1,1)A_i\diag(Q_1,1)^{-1}=
\begin{pmatrix}Q_1(A_i)_{\leq n-1}^{\leq n-1}Q_1^{-1}&Q_1\aaa_i\\ \bbb_i^TQ_1^{-1}&0\end{pmatrix}.
$$
If $(a_{n-1,1},a_{n-1,2})$ and $(b_{1,n-1}, b_{2,n-1})$ are linearly dependent,
then
$(ta_{n-2,1}+a_{n-1,1},ta_{n-2,2}+a_{n-1,2})$ and $(b_{1,n-1}, b_{2,n-1})$ are linearly independent
for any $t\in\FFF\setminus\{0\}$
since
$(a_{n-2,1},a_{n-2,2})$ and $(a_{n-1,1}, a_{n-1,2})$ are linearly independent
by the choice of $Q_1$.
Choose $t\in\FFF\setminus\{0\}$ so that
$ta_{n-2,1}+a_{n-1,1}\neq 0$
and set 
$Q_2=E_{n-1}+tE_{n-1,n-2}$.
Then $P=Q_2Q_1$ meets our needs since
$$
\diag(Q_2Q_1,1)A_i\diag(Q_2Q_1,1)^{-1}=
\begin{pmatrix}Q_2Q_1(A_i)_{\leq n-1}^{\leq n-1}Q_1^{-1}Q_2^{-1}&Q_2Q_1\aaa_i\\
 \bbb_i^TQ_1^{-1}Q_2^{-1}&0\end{pmatrix}
$$
and
$Q_2^{-1}=E_{n-1}-tE_{n-1,n-2}$,
the
$(n-1,n)$ entry of 
$\diag(Q_2Q_1,1)A_i\diag(Q_2Q_1,1)^{-1}$ is $ta_{n-2,i}+a_{n-1,i}$
and
$(n,n-1)$ entry of
$\diag(Q_2Q_1,1)A_i\diag(Q_2Q_1,1)^{-1}$ is $b_{i,n-1}$.
Therefore we have proved the case where
$\rank(\aaa_1,\aaa_2)=2$.

We can prove the case where $\rank(\bbb_1,\bbb_2)=2$ by the same way.

Now assume that
$\rank(\aaa_1,\aaa_2)=\rank(\bbb_1,\bbb_2)=1$.
Choose as before, a non-singular
$(n-1)\times (n-1)$ matrix $Q_1$ such that any entry of
$Q_1\aaa_1$ and $\bbb_1^TQ_1^{-1}$
is not zero
and set
$Q_1(\aaa_1,\aaa_2)=(a_{ij})$,
$(\bbb_1,\bbb_2)^TQ^{-1}=(b_{ij})$.
Then $a_{n-1,2}/a_{n-1,1}\neq b_{2,n-1}/b_{1,n-1}$,
since otherwise
$-a_{n-1,2}/a_{n-1,1}\aaa_1+\aaa_2=
-b_{2,n-1}/b_{1,n-1}\aaa_1+\aaa_2=
-b_{2,n-1}/b_{1,n-1}\bbb_1+\bbb_2=\zerovec$,
contradicts the assumption.
Therefore $P=Q_1$ meets our needs.
\end{proof}
Now we state the following
\begin{thm}\label{thm:cont sing}
Let $T=(A_1;A_2;A_3)$ be an $n\times n\times 3$ tensor.
If $\langle A_1,A_2,A_3\rangle$ contains a non-zero singular matrix,
then
$\rank_\FFF T\leq 2n-1$.
In particular,
if $\FFF=\CCC$ or $n$ is odd,
then $\rank_\FFF T\leq 2n-1$.
\end{thm}
\begin{proof}
We prove by induction on $n$.

Since
$\maxrank_\FFF(1,1,3)=1$ and
$\maxrank_\FFF(2,2,3)=3$,
we may assume that $n\geq 3$.
By 
Lemma \ref{lem:rank basic1} and the assumption,
we may assume that $A_3=\diag(E_r,O)$ with $r<n$
and $\supp(A_1)\supset\supp(A_2)$.

If $(i,j)\in\supp(A_1)$ for some $(i,j)$ with $i>r$ and $j>r$,
by 
permuting
rows and columns within $(r+1)$-th, \ldots, $n$-th one, if necessary,
we can apply Lemma \ref{lem:use ab}.
Therefore $\rank_\FFF T\leq 2n-1$.

Now assume that
$(i,j)\not\in\supp(A_1)$ for
any $i$, $j$ with $i>r$ and $j>r$.
Set
${}_{r<}(A_i)^{\leq r}=A_{12i}$ and
${}^{r<}(A_i)_{\leq r}=A_{21i}$.
If there is a column vector of $A_{121}$ which is $\zerovec$,
then
$\rank_\FFF T\leq n+n-1$
by Lemma \ref{lem:AS:lemma4},
since $T$ is essentially an $n \times (n-1) \times 3$ tensor in this case.
Therefore we may assume that no column vector of $A_{121}$ is $\zerovec$.
We may also assume that no row vector of $A_{211}$ is $\zerovec^T$.

Set $A_{12i}=(\aaa_{i, r+1},\ldots, \aaa_{in})$
and $A_{21i}^T=(\bbb_{i, r+1},\ldots, \bbb_{in})$.
Assume first that there is $j>r$ such that $\aaa_{1j}$, $\aaa_{2j}$ are
linearly independent.
Then by exchanging the $(r+1)$-th and the $j$-th columns,
we may assume that
$(A_1)_{\leq r+1}^{\leq r+1}$
and
$(A_2)_{\leq r+1}^{\leq r+1}$
satisfy the condition of Lemma \ref{lem:prep as lemma3}.
So we take the non-singular $r\times r$ matrix $P$ of the conclusion of 
Lemma \ref{lem:prep as lemma3}
and set
$
\diag(P,E_{n-r})A_k\diag(P,E_{n-r})^{-1}=(a_{ijk})$.
Then $a_{r+1,r+1,k}=0$ for any $k$ and
$a_{r,r+1,2}/a_{r,r+1,1}\neq a_{r+1,r,2}/a_{r+1,r,1}$.
Therefore, 
by exchanging the $(r+1)$-th and the $n$-th rows and columns, 
and exchanging the $r$-th and the $(n-1)$-th rows and columns, 
if necessary,
we may transform
$\diag(P,E_{n-r})(A_1;A_2;A_3)\diag(P,E_{n-r})^{-1}$ 
to a tensor which satisfy the condition of
Lemma \ref{lem:use ab}
(we do not need the permutation if $r=n-1$).
So the conclusion follows by Lemma \ref{lem:use ab}.
The case that there is $j>r$ such that $\bbb_{1j}$, $\bbb_{2j}$ 
are linearly independent
is proved by the same way.

Next assume that $\aaa_{1j}$, $\aaa_{2j}$ are linearly dependent
and
$\bbb_{1j}$, $\bbb_{2j}$ are linearly dependent
for any $j>r$.

Since the 
vector space spanned by the 
column vectors of 
$(A_1)^{\leq r}_{\leq r+1}$
is at most 
$r$
and the last column of 
$(A_1)^{\leq r}_{\leq r+1}$
 is not zero,
we see that there is $j$ with $1\leq j\leq r$ such that 
$j$-th column of $(A_1)^{\leq r}_{\leq r+1}$ is a linear combination of
the columns of ${}_{j<}(A_1)^{\leq r}_{\leq r+1}$.
Therefore
we see that there is an 
$(r+1)\times (r+1)$ 
lower triangular unipotent matrix $V$ such that
$((A_1)_{\leq r+1}V)_{\leq r}^{\leq r}=((A_1)^{\leq r}_{\leq r+1}V)_{\leq r}$ 
has a column vector which is $\zerovec$.
So by the induction hypothesis,

\begin{eqnarray*}
&&\rank_\FFF T\\
&=&\rank_\FFF(A_1;A_2;A_3)\\
&=&\rank_\FFF(A_1\diag(V,E_{n-r-1});A_2\diag(V,E_{n-r-1});A_3\diag(V,E_{n-r-1}))\\
&\leq & \rank_\FFF(((A_1)_{\leq r+1}V)_{\leq r}^{\leq r};
((A_2)_{\leq r+1}V)_{\leq r}^{\leq r};
((A_3)_{\leq r+1}V)_{\leq r}^{\leq r})\\
&&
+\sum_{j=r+1}^n\rank_\FFF(\aaa_{1j};\aaa_{2j};\zerovec)
+\sum_{j=r+1}^n\rank_\FFF((\bbb_{1j}^T,0)V;(\bbb_{2j},0)^TV;\zerovec^T)\\
&\leq &2r-1+(n-r)+(n-r)\\
&=&2n-1
\end{eqnarray*}
since
$\rank_\FFF(\aaa_{1j}$; $\aaa_{2j};\zerovec)\leq 1$ and 
$\rank_\FFF(\bbb_{1j}^T;\bbb_{2j}^T;\zerovec^T)\leq1$
for any $j$ with $j>r$.
\end{proof}

Next, we consider the non-square case.
First we prepare the following lemmas.
\begin{lemma}\label{lem:prep non-square}
Let $A$ and $B$ be $n\times n$ matrices,
$\aaa=(a_1, \ldots, a_n)^T$ and
$\bbb=(b_1, \ldots, b_n)^T$
be $n$-dimensional vectors.
Suppose $a_i\neq 0$ for any $i=1$, \ldots, $n$.
Then there are diagonal matrices $X$ and $Y$ 
and a vector $\ppp$ such that
\begin{enumerate}
\item
$A+X$ is non-singular,
\item
$(A+X)\ppp=\aaa$ and $(B+Y)\ppp=\bbb$.
\end{enumerate}
Moreover, if $b_1/a_1$, \ldots, $b_n/a_n$ are distinct each other,
then we can take $X$ and $Y$ so that 
$(A+X)^{-1}(B+Y)$ has $n$ distinct eigenvalues in $\FFF$.
\end{lemma}
\begin{proof}
Set $A=(a_{ij})$ and $B=(b_{ij})$.
For $0<\epsilon\in\RRR$, we set
\begin{eqnarray*}
a_i(\epsilon)&=&a_i-\epsilon\sum_{j=1}^n a_{ij}\\
b_i(\epsilon)&=&b_i-\epsilon\sum_{j=1}^n b_{ij}\\
D_1(\epsilon)&=&\diag(a_1(\epsilon),\ldots,a_n(\epsilon))\\
D_2(\epsilon)&=&\diag(b_1(\epsilon),\ldots,b_n(\epsilon)).
\end{eqnarray*}
Then 
$$
(\epsilon A+D_1(\epsilon))\onevec=\aaa
\quad\mbox{and}\quad
(\epsilon B+D_2(\epsilon))\onevec=\bbb
$$
where $\onevec=(1,\ldots, 1)^T$.

By the same argument as in the proof of
Lemma \ref{lem:as lemma3},
we see that $\epsilon A+D_1(\epsilon)$ is non-singular if $\epsilon>0$ 
is sufficiently small and
if $b_1/a_1$, \ldots, $b_n/a_n$ are distinct each other,
we can take $\epsilon$ so that 
$(\epsilon A+D_1(\epsilon))^{-1}(\epsilon B+D_2(\epsilon))$ 
has $n$ distinct eigenvalues in $\FFF$.

Therefore, it is enough to set
$X=(1/\epsilon)D_1(\epsilon)$,
$Y=(1/\epsilon)D_2(\epsilon)$
and
$\ppp=\epsilon\onevec$.
\end{proof}
\begin{lemma}\label{lem:use non-square prep}
Let $(A_1;A_2)$ be an $m\times n\times 2$ tensor with $m<n$.
Set $A_i=(\aaa_{i1},\ldots, \aaa_{in})$ for $i=1,2$.
Suppose $(A_1)_{\leq m}$ is non-singular and
$((A_1)_{\leq m})^{-1}(A_2)_{\leq m}$ has $m$ distinct eigenvalues.
Suppose also that there are integers $j_1$, \ldots, $j_s$ with
$m<j_1<\cdots<j_s\leq n$ and
$m$-dimensional vectors $\ppp_1$, \ldots, $\ppp_s$ such that
$$
(A_i)_{\leq m}\ppp_t=\aaa_{ij_t}\quad
\mbox{for $i=1,2$, $t=1,2,\ldots, s$.}
$$
Then $\rank_\FFF(A_1;A_2)\leq n-s$.
\end{lemma}
\begin{proof}
Let $V$ be the $n\times n$ upper triangular unipotent matrix whose $j$-th column
is
$
\begin{pmatrix}-\ppp_t\\\zerovec\end{pmatrix}+\eee_{j_t}
$
if $j=j_t$ for some $t$
and $\eee_j$ otherwise.

Then $j_1$, $j_2$, \ldots, $j_s$-th column of $A_iV$ is zero by the assumption
and therefore we see by Lemma \ref{lem:AS:lemma4} that
$$
\rank_\FFF(A_1;A_2)=\rank_\FFF(A_1V;A_2V)\leq n-s,
$$
since  $(A_1V;A_2V)$ is essentially an $m\times (n-s)\times 2$ tensor.
\end{proof}
Now we state the following
\begin{thm}\label{thm:non-square}
If $m<n$ then
$\maxrank_\FFF(m,n,3)\leq m+n-1$.
\end{thm}
\begin{proof}
We prove for an arbitrary $m\times n\times 3$ tensor 
$T=(A_1;A_2;A_3)$,
$\rank_\FFF T\leq m+n-1$.

Set $r=\max\{\rank A\mid A\in\langle A_1,A_2,A_3\rangle\}$.
Then by Lemma \ref{lem:rank basic1},
we may assume that $A_3=(\diag(E_r,O),O)$
and 
$\supp(A_1)\supset \supp(A_2)$.

Set
$A_i=(\aaa_{i1},\ldots, \aaa_{in})$ for $i=1,2$.
If there is $j>m$ such that $\aaa_{1j}=\zerovec$,
then, since we are assuming that $\supp(A_1)\supset\supp(A_2)$,
$T$ is essentially an $m\times (n-1)\times 3$ tensor.
So 
$\rank_\FFF T\leq m+n-1$
by Lemma \ref{lem:AS:lemma4}.

Now assume that $\aaa_{1j}\neq\zerovec$ for any $j>m$.

We first consider the case where $\aaa_{1j}$, $\aaa_{2j}$ are linearly dependent for any
$j$ with $j>m$.
Since the vector space spanned by the column vectors of $(A_1)_{\leq m+1}$ is at most $m$
and the last column of $(A_1)_{\leq m+1}$ is not zero,
we see that there is $j$ with $1\leq j\leq m$ such that $j$-th column vector of $A_1$
is a linear combination of the column vectors of 
${}_{j<}(A_1)_{\leq m+1}$.
Therefore we see that there is an $(m+1)\times (m+1)$ lower triangular unipotent matrix $V$
such that
$(((A_1)_{\leq m+1})V)_{\leq m}$ has a column vector which is $\zerovec$.
So we see by Theorem \ref{thm:cont sing}
\begin{eqnarray*}
&&\rank_\FFF T\\
&=&\rank_\FFF(
A_1\diag(V,E_{n-m-1});
A_2\diag(V,E_{n-m-1});
A_3\diag(V,E_{n-m-1}))\\
&\leq&\rank_\FFF(
((A_1)_{\leq m+1}V)_{\leq m};
((A_2)_{\leq m+1}V)_{\leq m};
((A_3)_{\leq m+1}V)_{\leq m})\\
&&\quad
+\sum_{j=m+1}^n\rank_\FFF(\aaa_{1j};\aaa_{2j};\zerovec)\\
&\leq& 2m-1+n-m\\
&=&m+n-1,
\end{eqnarray*}
since $\aaa_{1j}$, $\aaa_{2j}$ are linearly dependent for $j>m$.

From now on, we assume that there is $j$ with $j>m$ such that
$\aaa_{1j}$, $\aaa_{2j}$ are linearly independent.

We first consider the case where $r=m$.
By Lemma \ref{lem:make nonzero},
we see that there is a non-singular $m\times m$ matrix $P$ such that
any entry of $P\aaa_{1j}$ and any $2$-minor of $P(\aaa_{1j},\aaa_{2j})$ is not zero.
Set $B_i=PA_i\diag(P,E_{n-m})^{-1}$ and
$B_i=(\bbb_{i1},\ldots, \bbb_{in})$  for $i=1,2,3$.
Then $B_3=(E_m,O)$ and every entry of $\bbb_{1j}$ and every $2$-minor of $(\bbb_{1j},\bbb_{2j})$
is not zero.
So by Lemma \ref{lem:prep non-square}, we see that there are $m\times m$ diagonal matrices
$D_1$ and $D_2$ and an $m$-dimensional vector $\ppp$ such that
\begin{description}
\item[\quad] $((B_i)_{\leq m}+D_i)\ppp=\bbb_{ij}\quad\mbox{for $i=1,2$,}$
\item[\quad] $(B_1)_{\leq m}+D_1$ is non-singular and
\item[\quad] $((B_1)_{\leq m}+D_1)^{-1}((B_2)_{\leq m}+D_2)$ has $m$ distinct eigenvalues.
\end{description}
Therefore by Lemma \ref{lem:use non-square prep}, we see that
$$
\rank_\FFF(B_1+(D_1,O);B_2+(D_2,O))\leq n-1.
$$
So
\begin{eqnarray*}
&&\rank_\FFF T\\
&= &\rank_\FFF(B_1;B_2;B_3)\\
&\leq&\rank_\FFF(B_1+(D_1,O);B_2+(D_2,O))+\rank_\FFF(-(D_1,O);-(D_2,O);(E_m,O))\\
&\leq& n-1+m.
\end{eqnarray*}

Finally we consider the case where $r<m$.
Since $A_3=(\diag(E_r,O),O)$ and 
$\rank(tA_3+A_1)\leq r$ for any $t\in\FFF$ by the definition of $r$,
we see that $(i,j)\not\in\supp(A_1)$ if $i>r$ and $j>r$.

If the $(r+1)$-th row of $A_1$ is zero,
then $(A_1;A_2;A_3)$ is essentially an $(m-1)\times n\times 3$ tensor.
So
$$\rank_\FFF(A_1;A_2;A_3)\leq m-1+n
$$
by Lemma \ref{lem:AS:lemma4}.
Therefore we may assume
that $(r+1)$-th row of $A_1$ is not zero.
Take $j$ with $j>m$ such that $\aaa_{1j}$, $\aaa_{2j}$ are linearly independent.
Exchanging the $(r+1)$-th and the $j$-th columns of $A_i$,
we may assume that $\aaa_{1,r+1}$, $\aaa_{2,r+1}$ are linearly independent.
By applying Lemma \ref{lem:prep as lemma3} to 
$(A_1)_{\leq r+1}^{\leq r+1}$
and
$(A_2)_{\leq r+1}^{\leq r+1}$,
we see that there is a non-singular $r\times r$ matrix $P$ such that
$\diag(P,1)(A_1)_{\leq r+1}^{\leq r+1}\diag(P,1)^{-1}$ 
and
$\diag(P,1)(A_2)_{\leq r+1}^{\leq r+1}\diag(P,1)^{-1}$
satisfy the condition of (b) in Lemma \ref{lem:as lemma3}.
Set
$
B_i=\diag(P,E_{m-r})A_i\diag(P,E_{n-r})^{-1}$ 
for $i=1,2,3$.
Then
$B_3=(\diag(E_r,O),O)$ and,
$(B_1)_{\leq r+1}^{\leq r+1}$
and
$(B_2)_{\leq r+1}^{\leq r+1}$,
satisfy the condition (b) in Lemma \ref{lem:as lemma3}.

Let $C_i$ be the $m\times n$ matrix obtained by exchanging the 
$(r+1)$-th and $m$-th rows and columns 
and $r$-th and $(m-1)$-th rows and columns
of $B_i$ respectively
for $i=1,2,3$.
Then 
$(C_1)_{\leq m}$ and 
$(C_2)_{\leq m}$ 
satisfy the condition of (b) in Lemma \ref{lem:as lemma3}
and
$C_3=(\diag(E_{r-1},O,1,0),O)$.
Therefore
we see that
$$
\rank_\FFF T=\rank_\FFF(C_1;C_2;C_3)\leq m+n-1
$$
by Lemma \ref{lem:use ab}.
\end{proof}
%
%
%

Finally
we state some upper bounds of the maximal rank 
for small tensors
which are direct consequences of Theorem \ref{thm:cont sing}.

\begin{prop} \label{thm:smallsize}
The followings are true.
\begin{enumerate}
\item $\maxrank_{\FFF}(3,3,3)\leq 5$
\item $\maxrank_{\CCC}(4,4,3)\leq 7$
\item $\maxrank_{\FFF}(5,5,3)\leq 9$
\item $\maxrank_{\CCC}(6,6,3)\leq 11$
\end{enumerate}
\end{prop}

It is possible that there is no non-zero singular matrix in $\langle A_1,A_2,A_3\rangle$ over the real number field.
For example, let
$A_1=\begin{pmatrix} 0&1&0&0\cr -1&0&0&0\cr 0&0&0&1\cr 0&0&-1&0\end{pmatrix}$,
$A_2=\begin{pmatrix} 0&0&0&1\cr 0&0&1&0\cr 0&-1&0&0\cr -1&0&0&0\end{pmatrix}$
and $A_3=E_4$.  
Since the determinant of $xA_1+yA_2+zA_3$ is
$(x^2+y^2+z^2)^2$, $xA_1+yA_2+zA_3$ is singular only when $x=y=z=0$.




\end{document}